\documentclass[12pt,twoside,reqno]{amsart}

\newtheorem{thm}{Theorem}[section]
\newtheorem{lemma}{Lemma}[section]

\newtheorem{proposition}{Proposition}[section]

\def \R{{\Bbb R}}

\newcommand{\N}{{\mathbb N}}

\numberwithin{equation}{section}

\begin{document}

\title[Critical ZK equation]
{
Generalized ZK equation posed on  a half-strip
}

\author[
N.~A. Larkin]
{
	N.~A. Larkin \\
	Departamento de Matem\'atica,\\
	Universidade Estadual de Maringá\\
	87020-900, Maringá, Parana, Brazil}
\

\thanks{\it {Mathematics Subject Classification 2010: 35G61, 35Q53.} }

\thanks
{email:nlarkine@uem.br}

\keywords
{ZK equation, stabilization}

\begin{abstract}
An initial-boundary value problem for the  generalized 2D Zakharov-Kuznetsov equation posed on the right half-strip  is considered.
Existence, uniqueness and the  exponential decay rate of global regular solutions for small initial data are established.
\end{abstract}

\maketitle

\section{Introduction}\label{introduction}

We are concerned with an initial-boundary value problem (IBVP) 
  for the critical generalized Zakharov-Kuznetsov (ZK) equation  posed
  on the right half-strip
\begin{equation}
u_t+u^2u_x +u_{xxx}+u_{xyy}=0\label{zk}
\end{equation}
which is a two-dimensional analog of the   generalized Korteweg-de Vries (KdV) equation
\begin{equation}\label{kdv}
u_t+u^ku_x+u_{xxx}=0
\end{equation}
with  plasma physics applications \cite{zk} that has been intensively studied last years \cite{doronin1,familark,jeffrey,kaku}.                                  

Equations \eqref{zk} and \eqref{kdv} are typical examples
of so-called
dispersive equations attracting considerable attention
of both pure and applied mathematicians. The KdV
equation is  more studied in this context.
The theory of  initial-value problems
(IVP henceforth)
for \eqref{kdv} is considerably advanced today
\cite{tao,kato}.

 Although dispersive equations were deduced for the whole real line, necessity to calculate numerically the Cauchy problem approximating the real line either by finite or semi-finite intervals implies to study initial-boundary value problems posed on bounded and unbounded intervals \cite{doronin1,faminski2,familark,larluc2,temam,temam2}.
What concerns (1.2) with $k>1, \;l=1$, called generalized KdV equations, the Cauchy problem  was studied in \cite{ martel,merle}, where it has been established  that for $k=4$ (the critical case)  the problem is well-posed for small  initial data, whereas for arbitrary initial data solutions may blow-up in a finite time. The generalized Korteweg-de Vries equation was  studied for understanding the interaction between the dispersive term and the nonlinearity in the context of the theory of nonlinear dispersive evolution equations \cite{jeffrey,kaku,lar6}. \\
 Recently, due to physics and numerics needs, publications on initial-boundary value
problems in both bounded and unbounded domains for dispersive equations  have been appeared
\cite{larluc2,pastor,pastor2}. In
particular, it has been discovered that the KdV equation posed on a
bounded interval possesses an implicit internal dissipation. This allowed
to prove the exponential decay rate of small solutions for
\eqref{kdv} with $k=1$  posed on bounded intervals without adding any
artificial damping term \cite{doronin1}. Similar results were proved
for a wide class of dispersive equations of any odd order with one
space variable \cite{familark,larluc2}.

The interest on dispersive equations became to extend their study for 
multi-dimensional models such as Kadomtsev-Petviashvili (KP)
and ZK equations. We call (1.1) a critical ZK equation by analogy with the critical KdV equation (1.2) for $k=4.$ It means that we did not be able to prove the existence and uniqueness of global regular solutions without smallness restrictions for initial data similarly to the critical case for the KdV equation \cite{larluc2,martel,merle}.
As far as the ZK equation is concerned,
the results on both IVP and IBVP can be found in
\cite{faminski2,faminski3,pastor,pastor2,marcia}.
We must note that solvability of initial-boundary value problems in classes of global regular solutions for the regular case of the 2D ZK equation\; ($uu_x$)\; has been established in \cite{doronin1,faminski3, lar2,lar5,larkintronco,marcia,temam,temam2} for arbitrary smooth initial data. On the other hand, for the 3D ZK equation, the convective term\; $uu_x,$  which is regular for the 2D ZK equation, corresponds to a critical case. It means that to prove the existence and uniqueness of global regular solutions one must put restrictions of small initial data \cite{lar3,lar4,larpad}.

The main goal of our work is to prove  for small initial data the existence and uniqueness
of global-in-time regular solutions for \eqref{zk} posed  on the right half-strip
 and the exponential decay rate of
these solutions.

The paper is outlined as follows: Section I is the Introduction. Section 2 contains  formulation
of the problem and auxiliaries. In Section \ref{existence}, Galerkin`s approximations
are  used to prove the existence and uniqueness of regular solutions. In Section 4,  decay of these solutions   is established.

\section{Problem and preliminaries}\label{problem}

Let $(x,y)\equiv\;(x_1,x_2)\in\Omega$ and $\Omega$ \;be a domain in $\R^2$.
We use the usual notations of Sobolev spaces $W^{k,p}$, $L^p$ and $H^k$   and the following notations for the norms \cite{Adams}:

$$\| f \|_{L^p(\Omega)}^p = \int_{\Omega} | f  |^p\, d\Omega,\;\;
\| f \|_{W^{k,p}(\Omega)} = \sum_{0 \leq | \alpha| \leq k} \|D^\alpha f \|_{L^p(\Omega)},\;p\in(1,+\infty).$$
$$ \|f\|_{L^{\infty}(\Omega)}=ess\; \sup_{\Omega}|f(x,y)|;\;\;W^{k,2}(D)=H^k(D).$$

Let $B$ be a positive number. Define
\begin{align*}
&D=\{(x,y)\in\mathbb{R}^2: \ x>0,\ y\in(0,B) \},\ \ \ Q=D\times \R^+;\\&
\gamma=\partial D \; \text{is a boundary of } \;D.
\end{align*}

Consider the following IBVP:
\begin{align}
Lu&\equiv u_t+u^2u_x+u_{xxx}+u_{xyy}=0\ \ \text{in}\;\; Q;
\label{2.1}
\\
&u_{\gamma\times t} =0,\; t>0;
\label{2.2}
\\
&u(x,y,0)=u_0(x,y),\ \ (x,y)\in D,
\label{2.4}
\end{align}
where $u_0:D\to\mathbb{R}$ is a given function.

Hereafter subscripts $u_x,\ u_{xy},$ etc. denote the partial derivatives,
as well as $\partial_x$ or $\partial_{xy}^2$ when it is convenient.
Operators $\nabla$ and $\Delta$ are the gradient and Laplacian acting over $D.$
By $(\cdot,\cdot)$ and $\|\cdot\|$ we denote the inner product and the norm in $L^2(D),$
and $\|\cdot\|_{H^k(D)}$ stands for the norm in $L^2$-based Sobolev spaces.

We will need the following result \cite{lady2}.
\begin{lemma}\label{lemma1}
Let $u\in H^1(D)$ and $\gamma$ be the boundary of $D.$

If $u|_{\gamma}=0,$ then
\begin{equation}\label{2.5}
\|u\|_{L^q(D)}\le \beta\|\nabla u\|^{\theta}\|u\|^{1-\theta}.
\end{equation}
We will use frequently the following inequaliies:
$$\|u\|_{L^4(D)}\leq 2^{1/2}\|\nabla u\|^{1/2}\|u\|^{1/2},\;\;\|u\|_{L^8(D)}\leq 4^{3/4}\|\nabla u\|^{3/4}\|u\|^{1/4}.$$ 

If $u|_{\gamma}\ne0,$ then
\begin{equation}\label{2.6}
\|u\|_{L^q(D)}\le C_{D}\|u\|^{\theta}_{H^1(D)}\|u\|^{1-\theta},
\end{equation}
where $  \theta =2(\frac{1}{2}-\frac{1}{q}).$ 
\end{lemma}

\begin{lemma} \label{steklov} Let $v \in H^1_0(0,B).$ Then
	\begin{equation}\label{Estek} 
	\|v_y\|^2\geq \frac {\pi^2}{B^2}\|v\|^2.
	\end{equation}
\end{lemma}

\begin{proof} The proof is based on the Steklov inequality \cite{steklov}: let $v(t)\in H^1_0(0,\pi)$, then by the Fourier series $\int_0^{\pi}v_t^2(t)\,dt\geq\int_0^{\pi}v^2(t)\,dt.$
	Inequality \eqref{Estek} follows by a simple scaling.
\end{proof}

\begin{proposition}\label{prop1}
	Let for a.e. fixed $t$  $u(x,y,t)\in H^1(D)$ and $u_{xy}(x,y,t)\in L^2({D}).$ Then
	\begin{align}
		&\sup_{(x,y)\in{D}}u^2(x,y,t)\le 2\Big[ \|u\|^2_{H^1({D})}(t)+\|u_{xy}\|^2_{L^2({D})}(t)\Big]\notag\\&\leq 2\|u\|^2(t)_{H^2(D)}.
	\end{align}
\end{proposition}
\begin{proof}
	For a fixed $x\in (0,L)$ and for any $y\in (0,B),$ it holds
	$$
	u^2(x,y,t)=\int_{0}^y\partial_su^2(x,s,t)\,ds\le \int_{0}^Bu^2(x,y,t)\,dy+\int_{0}^Bu_y^2(x,y,t)\,dy$$
	$$\equiv\rho^2(x,t).
	$$
	On the other hand,
	$$
	\sup_{(x,y)\in\mathcal{D}}u^2\le \sup_{x\in(0,L)}\rho^2(x)=\sup_{x\in (0,L)}\left|\int_0^x\partial_s\rho^2(s)\,ds\right|$$$$
	\le2\int_0^L\int_{0}^B\left(u^2+u_x^2+u_y^2+u_{xy}^2\right)\,dx\,dy \leq 2\|u\|^2_{H^2(D)}.
	$$
	The proof of Proposition 2.1 is complete.
\end{proof}

\section{Existence theorem}\label{existence}

\begin{thm}\label{theorem1}
Given  $u_0(x,y)$  such that $u_0|_{\gamma}=0$ \; and

\begin{align}&J(u_0)\equiv\int_D(1+x)^2\Big[u^2_0(x,y)+|\nabla u_0(x,y)|^2+|\Delta u_{0x}(x,y)|^2\notag\\&+ u^4_0(x,y)u^2_{0x}(x,y)\Big]dxdy<\infty,
	\end{align}

\begin{equation}
\|u_0\|<min(\frac{1}{8},\frac{\pi^2}{4B^2}),\;\;K(0)<\frac{\pi^2}{2B^2},
\end{equation}
where
\begin{align*}
	&K(t)\equiv 2^8\|(1+x)u_0\|^2\Big(\|3(1+x)u_0\|^2+2((1+x)^2,u^2_t)(t)\Big)\\&
	+2^9\|(1+x)u_0\|\Big(5\|(1+x)u_0\|^3+4\|(1+x)u_t\|^3(t)\Big)\Big[1\\&
	+2^8\|(1+x)u_0\|\Big(5\|(1+x)u_0\|^3+4\|(1+x)u_t\|^3(t)\Big)\Big],
\end{align*}

$$((1+x^2,u^2_t)(0)=((1+x),^2\{u_{0xxx}+u_{0xyy}+u^2_0u_{0x}\}^2).$$
Then  there exists a unique strong solution to
\eqref{2.1}-\eqref{2.4} such that
\begin{align*}
&u\in L^{\infty}(\R^+;H^2(D));\;\Delta u_x\in L^{\infty}(\R^+;L^2(D));\\
&u_t\in L^{\infty}(\R^+;L^2(D))\cap L^2(\R^+;H^1(D)).
\end{align*}

\end{thm}
\begin{proof}
To prove this theorem, we will use the Faedo-Galerkin approximations. Let $w_j(y)$ be orthonormal in $L^2(D)$ eigenfunctions to the following Dirichlet Problem:
\begin{equation}
w_{jyy}+\lambda_jw_j=0, \;\;y\in (0,B); \;\; w_j(0)=w_j(B)=0;\;\;j\in {\N}.
\end{equation}
Define approximate solutions of (2.1)-(2.3) in the form:
\begin{equation}
u^N(x,y,t)=\sum_{j=1}^N g^N_j(x,t)w_j(y).
\end{equation}
Here $g^N_j(x,t)$ are solutions to the following Korteweg-de Vries system:
\begin{align}
&g^N_{jt}+ g^N_{jxxx}-\lambda_jg^N_{jx}+\int_0^B |u^N|^2u^N_x w_j(y) dy=0,\\
&g^N_j(0,t)=0;\; t>0,\\& g^N_j(x,0)=(u^N_{0},w_j),\;\;x\in \R^+,\;j=1,...,N,
\end{align}
where $u^N_{0}=\sum_{i=1}^N\alpha_{iN}w_i \; \text{and}\; \lim_{N\to\infty} J(u_0^N)= J(u_0).$

Since each regularized KdV equation from (3.5) is not critical, it is known \cite{lar6} that there exists a unique regular solution of (3.5)-(3.7) at least locally on time.

\noindent Our goal is to obtain global in $t$  a priori estimates for the $u^N$  independent of $t$ and $N,$ then to pass the limit as $N$ tends to $\infty$ getting a solution to (2.1)-(2.3).

{\bf Estimates of approximate solutions.}\\
{\bf Estimate I.}
Multiply (3.5) by $g^N_j$, sum up over $j=1,...,N$ and integrate over $\Omega \times (0,t)$ to obtain

\begin{align} &\| u^N \|^2(t) + \int_0^t \int_0^B (u_{x}^N)^2(0,y,\tau)\, dy \,  \, d\tau  \notag\\&= \|u_{0}^N \|^2\leq \|u_0\|^2,\;\;t>0.
\end{align}

{\bf Estimate II.}\label{2-nd estimate}
Write the inner product
$$2\left(Lu^N,(1+x)u^N\right)(t)=0,
$$
dropping the index $N$, in the form: 
\begin{align*}
\frac{d}{dt}\left((1+x),u^2\right)(t)
&+\int_{0}^Bu_x^2(0,y,t)\,dy
+3\|u_x\|^2(t)+\|u_y\|^2(t)\\
&=\frac12\int_{\mathcal{D}}u^4\,dx\,dy.
\end{align*}
Taking into account \eqref{2.5} and (3.8), we obtain
\begin{align*}
\frac12\int_{\mathcal{D}}u^4\,dx\,dy
&\le \frac12\|u\|^4_{L^4(\mathcal{D})}(t)
\le 2\|\nabla u\|^2(t)\|u_{0}\|^2(t).
\end{align*}
This implies
\begin{align}
\frac{d}{dt}&((1+x),u^2)(t)
+\frac12\|\nabla u\|^2(t)+\bigl(\frac12-2\|u_0\|^2\bigr)\|\nabla u\|^2(t)\notag\\&+2\|u_x\|^2(t)+\int_{0}^Bu_x^2(0,y,t)\,dy\le 0\end{align}

and consequently, due to (3.2),

\begin{align}\label{3.6}
\left((1+x),|u^N|^2\right)(t)+&\int_0^t\int_{0}^B|u^N_{x}|^2(0,y,\tau)\,dy\,d\tau+\frac{1}{2}\int_0^t\|\nabla u^N\|^2(\tau)\,d\tau
\notag\\&\le ((1+x),u_0^2), \;\;t>0.
\end{align}
Moreover, we can  rewrite (3.9) as
\begin{align}
&4\|u^N_x\|^2(t)+\|\nabla u^N\|^2(t)+\int_{0}^B|u^N_x|^2(0,y,t)\,dy\leq 4|((1+x)u^N,u^N_t)(t)|\notag\\ &\leq4\|(1+x)^{1/2}u^N\|(t)\|(1+x)^{1/2}u^N_t\|(t).
\end{align}	

{\bf Estimate III}\\

Write the inner product
$$2\left(Lu^N,(1+x)^2u^N\right)(t)=0,
$$
dropping the index $N$ and making use of (2.6) with respect the variable $y$, in the form: 
\begin{align*}
	&\frac{d}{dt}\left((1+x)^2,u^2\right)(t)
	+\int_{0}^Bu_x^2(0,y,t)\,dy
	+6\|(1+x)^{1/2}u_x\|^2(t)\\&+\|(1+x)^{1/2}u_y\|^2(t) +\frac{\pi^2}{B^2}\|(1+x)^{1/2}u\|^2(t)
	=-((1+x),u^4)(t)\\&\leq\|(1+x)^{1/2}u\|^4_{L^4(D)}(t)\notag\\&
\leq 4\|(1+x)^{1/2}u\|^2(t)\|(1+x)^{1/2}\nabla u+\frac{1}{2(1+x)^{1/2}}u\|^2(t)\notag\\&
\leq 4\|u_0\|^2\|(1+x)^{1/2}u\|^2(t)+4\|(1+x)^{1/2}u_0\|^2(t)\|(1+x)^{1/2}\nabla u\|^2(t).
\end{align*}

Making use of conditions of Theorem 3.1, we get
\begin{align*}
	&\frac{d}{dt}\left((1+x)^2,u^2\right)(t)
	+5\|(1+x)^{1/2}u_x\|^2(t)
	+\frac{1}{2}\|(1+x)^{1/2}\nabla u\|^2(t)\leq 0.
\end{align*}

This implies

\begin{align}
	&((1+x)^2,u^2)(t)+\int^t_{0}\{\|(1+x)^{1/2}u_x\|^2(\tau)
	+\frac{1}{2}\|(1+x)^{1/2}\nabla u\|^2(\tau)\}d\tau \notag\\
	&+((1+x)^2,u^2)(t)\leq ((1+x)^2,u^2)(0)\end{align}

and

\begin{align}
	&\|(1+x)^{1/2}u_x\|^2(t)
	+\frac{1}{2}\|(1+x)^{1/2}\nabla u\|^2(t)\notag\\& \leq 4\|(1+x)u_t\|(t)\|(1+x)u\|(t).
\end{align}

{\bf Estimate IV}\\
Dropping the index $N$, write the inner product
$$
2\left((1+x)^2u^N_t,\partial_t(Lu^N\right)(t)=0
$$
as
\begin{align}\label{3.13}
	\frac{d}{dt}
	&\left((1+x)^2,u_t^2\right)(t)+\int_{0}^Bu_{xt}^2(0,y,t)\,dy+6\|(1+x)^2u_{xt}\|^2(t)\notag\\&	+2\|(1+x)^2u_{yt}\|^2(t)\notag\\
	&=2\left((1+x)^2u^2u_t,u_{xt}\right)(t)+4((1+x)u^2,u_t^2)(t).
\end{align}

Making use of Lemmas 2.1, 2.2, (3.13) and taking into account the first inequality of (3.2), we estimate
\begin{align*}
	I_1
	&=2\left((1+x)^2u^2u_t,u_{xt}\right)(t)\\&\leq 2\|(1+x)^{1/2}u_{xt}\|(t)\|(1+x)^{1/2}u\|^2(t)_{L^8(D)}\|(1+x)^{1/2}u_t\|(t)_{L^4(D)}\\
	&\le \|(1+x)^{1/2}u_{xt}\|^2(t)+\frac{1}{2}\|(1+x)^{1/2}\nabla u_t\|^2\\&
	+2^9\|(1+x)^{1/2}u\|(t)\Big(\|u_0\|^3+\|(1+x)^{1/2}\nabla u\|^3(t)\Big)\Big[1\\&
	+2^8\|(1+x)^{1/2}u\|(t)\{\|u_0\|^3+\|(1+x)^{1/2}\nabla u\|^3(t)\}\|\Big](1+x)^{1/2}u_t\|^2(t)\\&
	\le \|(1+x)^{1/2}u_{xt}\|^2(t)+\frac{1}{2}\|(1+x)^{1/2}\nabla u_t\|^2\\&
	+2^9\|(1+x)u\|(t)\Big(\|u_0\|^3+2^3\|(1+x)u_t\|^{3/2}(t)\|(1+x)u\|^{3/2}(t)\Big)\Big[1\\&
	+2^8\|(1+x)u\|(t)\{\|u_0\|^3+2^3\||(1+x)u_t\|^{3/2}(t)\}\|(1+x)u_t\|^2(t)\\&
	\le \|(1+x)^{1/2}u_{xt}\|^2(t)+\frac{1}{2}\|(1+x)^{1/2}\nabla u_t\|^2\\&
	+2^9\|(1+x)u\|(t)\Big(5\|(1+x)u_0\|^3+4\|(1+x)u_t\|^3(t)\Big)\Big[1\\&
	+2^8\|(1+x)u\|(t)\{5\|(1+x)u_0\|^3+4\|(1+x)u_t\|^3(t)\Big]\|(1+x)u_t\|^2(t).
		\end{align*}
	
Similarly,
\begin{align*}
	I_2
	&=4((1+x)u^2,u_t^2)(t)\le 4\|(1+x)^{1/2}u\|^2(t)_{L^4(D)}\|u_t\|^2(t)_{L^4(D)}\\
	&\leq 2^4\|(1+x)^{1/2}u\|(t)\|\nabla ((1+x)^{1/2} u)\|(t)\|u_t\|(t)\|\nabla u_t\|(t)\\
	&\leq\frac{1}{2}\|\nabla u_t\|^2(t)+2^8\|(1+x)u\|^2(t)\Big(\|u_0\|^2+\|(1+x)^{1/2}\nabla u\|^2(t)\Big)\|u_t\|^2(t)\\
	&
	\leq\frac{1}{2}\|\nabla u_t\|^2(t)+2^8\|(1+x)u_0\|^2\Big[\|(1+x)u_0\|^2\\&+\|(1+x)\nabla u\|^2(t)\Big]\|(1+x)u_t\|^2(t)\\
	&
	\leq\frac{1}{2}\|\nabla u_t\|^2(t)+2^8\|(1+x)u_0\|^2\Big[\|3(1+x)u_0\|^2\\&+2\|(1+x) u_t\|^2(t)\Big]\|(1+x)u_t\|^2(t).
\end{align*}
Substituting $I_1, I_2$ into (3.14), making use of Lemma 2.2 and taking into account inequalities of (3.2), we come to the inequality
\begin{align}\label{3.14}
&\frac{d}{dt}
((1+x)^2,u_t^2)(t) +\int_{0}^Bu_{xt}^2(0,y,t)dy+4\|(1+x)^{1/2}u_{xt}\|^2(t)\notag\\&	+\frac{1}{2}\|(1+x)^{1/2}\nabla u_{t}\|^2(t)
 + \Big[\frac{\pi^2}{2B^2}-K(t)\Big]((1+x),u^2_t)(t)\leq 0,
\end{align}

where

\begin{align*}
	&K(t)\equiv 2^8\|(1+x)u_0\|^2\Big(\|3(1+x)u_0\|^2+2((1+x)^2,u^2_t)(t)\Big)\\&
		+2^9\|(1+x)u_0\|\Big(5\|(1+x)u_0\|^3+4\|(1+x)u_t\|^3(t)\Big)\Big[1\\&
		+2^8\|(1+x)u_0\|\Big(5\|(1+x)u_0\|^3+4\|(1+x)u_t\|^3(t)\Big)\Big]
\end{align*}

Since, by the conditions of Theorem 3.1, 
$$\frac{\pi^2}{2B^2} > K(0),$$
then, using (3.2), Lemma 2.2 and standard arguments,\;\cite{larluc2,larpad}, we obtain that

$$\frac{\pi^2}{2B^2} > K(t), \;t>0.$$

Returning to (3.15) and (3.14), we obrain

\begin{align}\label{3.14}
	&
	((1+x)^2,u_t^2)(t) +\int_0^t\{\int_{0}^Bu_{x\tau}^2(0,y,\tau)dy\notag\\&	+\frac{1}{2}\|(1+x)^{1/2}\nabla u_{\tau}\|^2(\tau)\}d\tau\leq ((1+x)^2,u_t^2)(0),
\end{align}
\begin{align}
	&2\|(1+x)^{1/2}u_x\|^2(t)
	+\|(1+x)^{1/2}\nabla u\|^2(t)\notag\\& \leq 8\|(1+x)u_t\|(t)\|(1+x)u\|(t)\notag\\&\leq 8\|(1+x)u_t\|(0)\|(1+x)u\|(0),
\end{align}
where

$$((1+x^2,u^2_t)(0)=((1+x),^2\{u_{0xxx}+u_{0xyy}+u^2_0u_{0x}\}^2).$$

{\bf Estimate V}\\

 Multiplying $j$-th equation of (3.5) by $\lambda_j$, and summing up the results over $j=1,...,N$, dropping the index $N$, we transfom the inner product
$$
-2\left((1+x)\partial^2_yu^N_,Lu^N\right)(t)=0
$$
into the inequality
\begin{align}\label{3.7}
&3\|u_{xy}\|^2(t)+\|u_{yy}\|^2(t)+\int_{0}^{B}u_{xy}^2(0,y,t)\,dy
+\notag\\
&=\frac23\left((1+x)(u^3)_x,u_{yy}\right)(t)+2((1+x)u_{yy}, u_t)(t)\notag\\
&=-\frac23\left((1+x)(u^3)_{yx},u_{y}\right)(t)+2((1+x)u_{yy}, u_t)(t)\notag\\
&\leq \delta\|\nabla u_{y}\|^2(t)+\frac{1}{\delta}\|(1+x)u_t\|^2(t)\notag\\
&+2(u^2,u^2_y)(t)+2(1+x)u^2u_y,u_{xy})(t).
\end{align}
Making use of Lemmas 2.1, 2.2,  we estimate
\begin{align*}
I_1&\equiv 2(u^2,u^2_y)(t)\leq2\|u\|^2_{L^4(D)}(t)\|u_y\|^2_{L^4(D)}(t)\notag\\
&\leq 4\|u\|(t)\|\nabla u\|(t)C^2_D\|u_y\|(t)\|\nabla u_y\|(t)\notag\\
&\leq \delta \|\nabla u_y\|^2(t)+\frac{4C^4_D}{\delta}\|u\|^2(t)\|\nabla  u\|^4(t),
\end{align*}
where $\delta$ is an arbitrary positive number,
\begin{align*}
I_2&\equiv 2(1+x)u^2u_y,u_{xy})(t)\leq 2\|u_{xy}\|(t)\|(1+x)u^2\|_{L^4(D)}(t)\|u_y\|_{L^4(D)}\\
&\leq 2^{3/2}C_D^{1/2}\|u_{xy}\|(t)\|u_y\|^{1/2}(t)\|\nabla u_y\|^{1/2}(t)\|(1+x)^{1/2}u\|^2_{L^8(D)}\\&
\leq\frac{1}{2}\|u_{xy}\|^2(t)+2^5C^2_D\|\nabla u_y\|(t)\|u_y\|(t)\|(1+x)^{1/2}u\|(t)\|(1+x)^{1/2}\nabla u\\&+\frac{u}{2(1+x)^{1/2}}\|^3(t)
\leq \frac{1}{2}\|u_{xy}\|^2(t)+\delta\|\nabla u_{y}\|^2(t)\\&
+\frac{2^{11}C^4_D}{\delta}\|u_y\|^2(t)\|(1+x)^{1/2}u\|^2(t)\Big(\|u_0\|^6+\|(1+x)^{1/2}\nabla u\|^6(t)\Big)\\&
\leq\frac{1}{2}\|u_{xy}\|^2(t)+\delta\|\nabla u_{y}\|^2(t)
+\frac{2^{11}C^4_D}{\delta}\frac{B^2}{\pi^2}\|(1+x)u_y\|^4(t)\Big(\|u_0\|^6\\&+2^5\|(1+x)u_t\|^3(0)\|(1+x)u\|^3(0)\Big).
\end{align*}
Substituting $I_1 ,I_2$ into (3.18), taking $\delta=\frac{1}{8},$ and making use of (3.12), (3.14),  we transform it into the following inequality:
\begin{align}
&\frac{7}{4}\| u_{xy}\|^2(t)+\frac12\|\nabla u_y\|^2(t)+\int_{0}^{B}u_{xy}^2(0,y,t)\,dy
\leq C_1\|(1+x)u_t\|^2(t)\notag\\&\leq C_1\|(1+x)u_t\|^2(0).   
\end{align}
where $C_1$ depends on $\|u_0\|, \|(1+x)u_t\|(0).$ Returning to Proposition 2.1 and making use of (3.9), (3.16), (3.17,  (3.19), we rewrite (2.7) as follows:

	\begin{align}
	&\sup_{(x,y)\in{D}}u^2(x,y,t)\le 2\Big[ \|u\|^2_{H^1({D})}(t)+\|u_{xy}\|^2_{L^2({D})}(t)\Big]\notag\\&\leq C(\|u_0\|, \|(1+x)u_t\|(0))\equiv C^2_s.
\end{align}

Obviously,
\begin{equation}\lim_{\|u_0\|_{H^2(D)}\to 0}C_s=0.
\end{equation}

{\bf Estimate VI}\\

Consider the following equation:
$$ -2((1+x)u^N_{xx},Lu^N)(t)=0$$
that, dropping the index $N$, can be rewritten in the form
\begin{align}& \|u_{xx}\|^2(t)+\int_0^B u^2_{xx}(0,y,t)dy-\|u_{xy}\|^2(t)-\int_0^B u^2_{xy}(0,y,t)dy\notag\\&
=2((1+x)u_t,u_{xx})(t)+2((1+x)u^2u_x,u_{xx})(t).
\end{align}

We estimate
\begin{align*} &I_1=2((1+x)u^2u_x,u_{xx})(t)\leq \frac14 \|u_{xx}\|^2(t)+4ess \sup_D u^2(x,y,t)\|u_x\|^2(t),\\&
	I_2=2((1+x)u_t,u_{xx})(t)\leq \frac14 \|u_{xx}\|^2(t)+4\|(1+x)u_t\|^2(t).
\end{align*}

Substituting $I_1, I_2$ into (3.22) and using (3.20), (3.21), we get

\begin{align}
&\|u^N_{xx}\|(t)+\int_0^B u^{N2}_{xx}(0,y,t)dy\leq C \|(1+x)u^N_t\|^2(t)\notag\\&\leq C\|(1+x)u_t\|^2(0).
\end{align}

This ineqaulity, (3.9) and (3.19) imply that
\begin{equation} 
\|u^N\|_{H^2(D)} \in L^{\infty}(\R^+;L^2(D))
\end{equation}
 uniformly in $N.$

Taking into account (3.3),(3.4), write (3.5) in the form
$$\int_0^B\bigl(u^N_{xxx}-\lambda_j u^N_x\bigr)w_j dy=-\int_0^B\bigl[u^N_t +|u^N|^2u^N_x\bigr]w_jdy.$$

Multiplying it by $g_{jxxx}-\lambda_j g_{jx}$, summing over $j=1,...,N$ and integrating with respect to $x$ over $(0,L),$ we obtain
$$\|\Delta u^N_x\|^2(t)\leq \|u^N_t\|(t)\|\Delta u^N_x\|(t)+\||u^N|^2u^N_x\|(t)\|\Delta u^N_x\|(t)$$
or
$$\|\Delta u^N_x\|(t)\leq \|u^N_t\|(t))+\||u^N|^2u^N_x\|(t).$$
Making use of  (3.16), (3.19), (3.20),  we get
\begin{equation}
\|\Delta u^N_x\|(t)\leq C,
\end{equation}
where the constant $C$ does not depend on $N,t>0.$

\subsection{Passage to the limit as $N\to \infty$}\label{limit}

Since the constants in (3.16),  (3.24), (3.25) do not depend on
$N,t>0,$ then, making use of the standard arguments, see \cite{temam2}, one may pass to the limit as $N\to \infty$ in (3.5) to obtain for all $\psi(x,y)\in L^2(D):$
\begin{equation}\label{3.22}
\int_{\mathcal{D}}\left[u_t+u^2u_x+\Delta u_x\right]\psi\,dx\,dy=0
\end{equation}
and establish the following result:

\begin{lemma}\label{lema1}
Let all the conditions of Theorem \ref{theorem1} hold. Then there exists a strong solution $u(x,y,t)$ to \eqref{2.1}-\eqref{2.4} such that
\begin{align}\label{3.23}
&\| u\|_{H^2(D)}^2(t)+\|u_t\|^2(t)+\|\Delta u_x\|^2(t)+\int_0^t\|\Delta u_x\|^2(\tau)d\tau\notag\\&+\|u_{x}(0,y,t)\|^2_{H^1_0(0,B)}\leq C, \;\;t>0.
\end{align}

\end{lemma}

This  completes the existence part of Theorem 3.1.

{\bf Uniqueness}

\begin{lemma} The regular solution from Lemma 3.1 is uniquelly defined.
	\end{lemma}

\begin{proof}
Let $u_1$ and $u_2$ be two distinct solutions to
\eqref{2.1}-\eqref{2.4}. Then $z=u_1-u_2$ solves the
following IBVP:
\begin{align}
Az&\equiv z_t++z_x+\frac12(u_1^3-u_2^3)_x+\Delta z_x=0\ \ \text{in}\ {Q},\;\;t>0,\label{5.1}\\
&z(0,y,t)=z(x,0,t)=z(x,B,t)\notag\\&=z(x,y,0)=0,\ \ (x,y)\in{D}.\label{5.3}
\end{align}
From the scalar product $$2\left(Az,(1+x)z\right)(t)=0,$$ we infer
\begin{align}
\frac{d}{dt}&((1+x),z^2)(t)(t)+3\|z_x\|^2(t)+\|z_y\|^2(t)+\int_{0}^Bz_x^2(0,y,t)\,dy\notag\\
&=-2((1+x)(u^3_1-u^3_2)_x,z)(t)=2((u_1^2+u_1u_2+u_2^2)z,z\notag\\&+(1+x)z_x)(t)\leq 2\|z_x\|^2(t)+C(M)((1+x),z^2)(t),
\end{align}

where $M=\sup_D|u_1^2+u_1u_2+u_2^2|(x,y,t)$. Due to Proposition 2.1 and (3.27), $M$ does not depend on $t>0.$ Hence (3.30) becomes
$$\frac{d}{dt}((1+x),z^2)(t)(t)\leq C(M)((1+x),z^2)(t).$$
Since $z(x,y,0)\equiv 0,$ by the Gronwall lemma,$$\|z\|^2(t)\leq ((1+x),z^2)(t)\equiv 0,\;\;t>0.$$
\end{proof}
The proofs of Lemma 3.2  and Theorem 3.1 are completed.\\
\end{proof}
\section{Decay of Regular Solutions}

\begin{thm}
	Let all the conditions of Theorem 3.1 be fulfilled and there exist $k>0$ and $C_s$  such that
	
	\begin{align} &C^2_s\leq min(\frac{k\pi^2}{4B^2},2k);\;\;k\leq \frac{\pi}{(20)^{1/2}B},\end{align}
	
	\begin{align}&\int_De^{kx}\Big[u^2_0(x,y)+|\nabla u_0(x,y)|^2+|\Delta u_{0x}(x,y)|^2\notag\\&+ u^4_0(x,y)u^2_{0x}(x,y)\Big]dxdy<\infty.
	\end{align}
Then the solution of Theorem 3.1 satisfies the following inequalities:		

\begin{align}
	&(e^{kx},u^2_t)(t)+\frac{k}{2}\int_0^t(e^{kx},|\nabla u_{\tau }|^2)(\tau)d\tau\leq (e^{kx},u^2_t)(0),
\end{align}

\begin{align}& \|u\|^2(t)_{H^2(D)}
	\leq  Ce^{(-\frac{k\pi^2}{2B^2}t)}.
\end{align}
\end{thm}
\begin{proof}
	
	{\bf Estimate 4.1}
For $k>0$, making use of (3.20), consider the following equation:
\begin{align}
	&2(Lu,e^{kx}u)(t)=\frac{d}{dt}(e^{kx},u^2)(t)+3k(e^{kx},u_x^2)(t)+k(e^{kx},u^2_y)(t)\notag\\&
+k\int_0^Bu^2_x(0,y,t)dy-k^3(e^{kx},u^2)(t)=\frac{k}{2}(e^{kx},u^4)(t)\notag\\&\leq\frac{k}{2}\sup_Du^2(x,y,t)(e^{kx},u^2)(t)\leq\frac{k}{2}C^2_s(e^{kx},u^2)(t).
\end{align}

Exploiting Lemma 2.2, transform (4.5) into the folowing inequality:

\begin{align}
	&\frac{d}{dt}(e^{kx},u^2)(t)+3k(e^{kx},u_x^2)(t)+\frac{k\pi^2}{2B^2}(e^{kx},u^2)(t)\notag\\&+\Big[\frac{k\pi^2}{2B^2}-k^3-\frac{k}{2}C^2_s\Big](e^{kx},u^2)(t)\leq 0.
	\end{align}

By conditions of Theorem 4.1, 
$$\Big[\frac{k\pi^2}{2B^2}-k^3-\frac{k}{2}C^2_s\Big]\geq 0$$

that implies

\begin{align*}
	&\frac{d}{dt}(e^{kx},u^2)(t)+\frac{k\pi^2}{2B^2}(e^{kx},u^2)(t)\leq 0
\end{align*}

and consequently,

\begin{align}(e^{kx},u^2)(t)\leq e^{(-\frac{k\pi^2}{2B^2}t)}(e^{kx},u^2)(0), \;\;t>0.
	\end{align}

{\bf Estimate 4.2}

\begin{align}
	&2((Lu)_t,e^{kx}u_t)(t)=\frac{d}{dt}(e^{kx},u^2_t)(t)+3k(e^{kx},u_{tx}^2)(t)+k(e^{kx},u^2_{ty})(t)\notag\\&
	+k\int_0^Bu^2_{tx}(0,y,t)dy-k^3(e^{kx},u^2_t)(t)=-2k(e^{kx}u^2,u^2_t)(t)\notag\\&-2(e^{kx}u^2u_t,u_{xt})(t)\leq
	2k\sup _D u^2(x,y,t)(e^{kx},u^2_t)(t)\notag\\&+2\sup _D u^2(x,y,t)(e^{kx}u_t,u_{xt})(t)\leq 2kC^2_s (e^{kx},u^2_t)(t)\notag\\&+2C^2_s\|(e^{kx/2}u_t\|(t)\|(e^{kx/2}u_{xt}\|(t).
	\end{align}

Making use of Lemma 2.2, we rewrite (4.8) in the form

\begin{align}
	&\frac{d}{dt}(e^{kx},u^2_t)(t)+(3k-C^2_s)(e^{kx},u_{tx}^2)(t)+\frac{k}{2}(e^{kx},u^2_{ty})(t)\notag\\&
	+\Big[\frac{k\pi^2}{2B^2}-k^3-2kC^2_s-C^2_s\Big]^3(e^{kx},u^2_t)(t)\leq 0.
\end{align}

There exist $k>0$ and $C_s$ such that the conditions of Theorem 4.1 imply

$$\frac{k\pi^2}{2B^2}-k^3-2kC^2_s-C^2_s\geq 0$$
and (4.9) becomes

\begin{align}
	&\frac{d}{dt}(e^{kx},u^2_t)(t)+k(e^{kx},u_{tx}^2)(t)+\frac{k}{2}(e^{kx},u^2_{ty})(t)\leq 0.
	\end{align}

Applying Lemma 2.2, we reduce (4.10) to the inequality

\begin{align*}
\frac{d}{dt}(e^{kx},u^2_t)(t)+\frac{k\pi^2}{2B^2}(e^{kx},u^2_{t})(t)\leq 0.
\end{align*}

Consequently,

\begin{equation}\|u_t\|^2(t)\leq (e^{kx},u^2_t)(t)\leq e^{(-\frac{k\pi^2}{2B^2}t)}(e^{kx},u^2_t)(0), \;\;t>0
	\end{equation}

and

\begin{align}
&(e^{kx},u^2_t)(t)+\frac{k}{2}\int_0^t(e^{kx},|\nabla u_{\tau }|^2)(\tau)d\tau\leq (e^{kx},u^2_t)(0).
\end{align}

{\bf Estimate 4.3}

Combining estimates (3.17), (3.20), (3.24), (4.11), we find that
\begin{equation} \|u\|^2(t)_{H^2(D)}\leq C\Big( \|(1+x)u_t\|^2(t)+\|(1+x)u\|^2(t)\Big).
\end{equation}

Consider 2 cases:\\{\bf case 1} $(k\geq 2)$. This implies that $1+x\leq e^{\frac{kx}{2}},$ hence by (4.7), (4.11),
\begin{align}
\|u\|^2_{H^2(D)}(t)\leq	C\Big[\|e^{\frac{kx}{2}}u_t\|^2(t)+\|e^{\frac{kx}{2}}u\|^2(t)\Big] \leq Ce^{(-\frac{k\pi^2}{2B^2}t)}.
\end{align}

{\bf case 2} $(k< 2)$. In this case, there is an interval $x\in [0,x_1],$ where $1+x\geq e^{\frac{kx}{2}}$ and there exists a function $\phi(x)\in (a_0,1]$ such that $ a_0>0$ and $\phi(x)(1+x)=e^{\frac{kx}{2}}.$ This allow us to rewrite (4.13) as

\begin{align}& \|u\|^2(t)_{H^2(D)}\leq C\Big( \|(1+x)u_t\|^2(t)+\|(1+x)u\|^2(t)\Big)\notag\\&
	=C\Big[((1+x)^2,u^2_t)(t)+((1+x)^2,u^2)(t)\Big]\notag\\&
	\leq C\int_0^B\Big[\int_0^{x_1}\phi^{-2}(x)e^{kx}[ u^2_t(x,y,t)+ u^2(x,y,t)]dx\notag\\&+\int_{x_1}^{\infty}e^{kx}[u^2_t(x,y,t)+u^2(x,y,t)]dx\Big]dy\notag\\&
	\leq C\frac{1}{a^2_0} \Big(\|e^{kx/2}u_t\|^2(t)+\|e^{kx/2}u\|^2(t)\Big)
	\leq  Ce^{(-\frac{k\pi^2}{2B^2}t)}.
\end{align}
This and (4.12) complete the proof of Theorem 4.1.

\end{proof}

{\bf Conclusions.} An initial-boundary value problem for the 2D critical  Zakharov-Kuznetsov equation posed on  a half-strip has been considered. Assuming small initial data, the existence of a global regular solution, uniqueness and exponential decay of  $\|u\|(t)_{H^2(D)}$ have been established.

{\bf Acknowledgements.}

{\bf Data Availability Statement.}
This research does not contain any data necessary to confirm.

\end{document}